\documentclass[12pt]{amsart}
\usepackage{amssymb}
\usepackage{amsfonts}

\newtheorem{theorem}{Theorem}[section]
\newtheorem{lemma}[theorem]{Lemma}
\theoremstyle{definition}
\newtheorem{definition}[theorem]{Definition}

\theoremstyle{remark}
\newtheorem{remark}[theorem]{Remark}
\numberwithin{equation}{section} \theoremstyle{plain}

\def\A{\mathbb A}
\def\C{\mathbb C}
\def\E{\mathbb E}
\def\F{\mathbb F}
\def\Q{\mathbb Q}
\def\R{\mathbb R}
\def\U{\mathbb U}
\def\V{\mathbb V}
\def\W{\mathbb W}
\def\T{\mathbb T}
\copyrightinfo{2001}{enter name of copyright holder}

\begin{document}

\title{Dynamics of Linear and Affine Maps}

\author{Ravi S. Kulkarni}

\address{Indian Institute of Technology (Bombay), Powai, Mumbai 400076, India,
\and Queens College and Gradaute Center, City University of New York.}

\email{punekulk@yahoo.com, kulkarni@math.iitb.ac.in}

\begin{abstract}

The well-known theory of  the ``rational canonical form of an operator"
describes the invariant factors, or equivalently, elementary divisors, as a
complete set of invariants of a  similarity class of an operator
on a finite-dimensional vector space $\V$ over a given field $\F$.
A finer part of the theory is the contribution by Frobenius
dealing with the structure of the centralizer of an operator. The
viewpoint is that of finitely generated modules over a PID, cf.
for example [J], ch. 3. In this paper we approach the issue from a
``dynamic" viewpoint. We also extend the
theory to affine maps. The formulation is in terms of the action
of the geneal linear group $GL(n)$, resp. the group of invertible affine
maps $GA(n)$, on the semigroup of all linear, resp. affine, maps by
conjugacy. The theory of rational canonical forms is connected 
with the orbits, and the Frobenius' theory with the orbit-classes, 
of the action of $GL(n)$ on the semigroup of linear maps. We describe a parametrization of orbits and
orbit-classes of both $GL(n)$- and $GA(n)$-actions, and also provide a
parametrization  of all affine maps themselves, which is
independent of the choices of linear or affine co-ordinate
systems, cf. sections 7, 8, 9. An important ingredient in these parametrizations is a
certain flag. For a linear map $T$ on $\V$, let  $Z_L(T)$  denote  its centralizer
associative $\F$-algebra, and $Z_L(T)^*$ the multiplicative  group of  invertible elements in $Z_L(T).$  In this
situation, we associate a canonical, maximal, $Z_L(T)$-invariant
flag, and precisely describe the orbits of $Z_L(T)^*$ on $\V,$ cf. section 3. 
Using this approach, we strengthen the classical theory in a number of ways.

\end{abstract}

\maketitle

\tableofcontents

\section{Introduction}

Let $\F$ be a field, and $\V$ an $n$-dimensional vector space over
$\F$. Let $L(\V)$ denote the set of all linear maps from $\V$ to $\V$.
Underlying $\V$ there is the affine space ${\A}$. Intuitively,
${\A}$ has no distinguished base-point which one can call as the``zero", 
or the ``origin". However
there is a well-defined  notion of ``difference of points". When
we distinguish a base-point $O$, and call it the zero, then there
is a well-defined notion of addition, making $\A$ into a vector
space.  An {\it affine map} of $\A$ is a  map   $(A, v): \V
\rightarrow \V$ of the form $(A, v)(x)= Ax + v,$ where $A$ is in
$L(\V)$, and $x, b$ are in V.  Then 
\begin{equation}\label{eq.1}
(A_1, v_1)\circ (A_2, v_2) = (A_1\circ A_2, A_1v_2 + v_1).
\end{equation}

\noindent This formula shows that $A(\V)$ is a semigroup with
identity under composition, and $L(\V)$ is a sub-semigroup of
$A(\V)$.

It is important to note that the representation $(A, v)$ depends on the
choice of the base-point. However the semigroup of affine maps,
and the form of an affine map is independent of this choice.
Indeed, let $O$ be a base-point making $\A$ into a vector space
$\V$. Let  $P$ be another point of $\A$ with the associated vector
$a$. Let $x$ resp $y$ be vector representations of a point $Q$
w.r.t. base-points $O$ and $P$. Then $y = x - a.$ Let $f$ be an
affine map of the form $(A, v)$ in the $x$-representation, and
$f(Q) = R$. Then the $x$-representation of $R$ is $Ax + v = Ay +
Aa + v$. So the $y$-representation of $R$ is $Ay + Aa + v - a = Ay
+ w$, where $ w = (A-I)a + v$. Hence the $y$-representation of $f$ is $(A, w).$ The maps induced by
the action of the group $(\V, +)$ on $\V$, called the {\it translations}, have the
form $\tau_a = (I, a)$. They form a subgroup $\T$, which is of
course isomorphic to $\V$. The above calculation shows that the
expression for {\hbox {$\tau_a: x \mapsto x + a$}} remains the
same no matter where we choose the base-point. In other words, ``$a$''
in $\tau_a$ has a {\it dynamic} as well as {\it affine} meaning. When $A
\neq I$, the same calculation shows that ``$A$" remains the same
no matter where we take the base-point, but ``$v$'' may change. In
other words, even when $A \neq I$, the ``$A$" has a dynamic
meaning, but ``$v$'' does not. The formula (1.1) shows that we have a
well-defined homomorphism $l: A(\V) \rightarrow L(\V)$ given by
$l((A, v))  = A$. We shall call $A$ the linear part of $(A, v)$.
We shall also call $v$ the translational part of $(A, v)$, with
the understanding that this specification depends on the choice of
the base-point. Note that the kernel of $l$, namely $l^{-1}(I)$
consists precisely of $\T$.

Let us also note an inconsistency in the usual terminology. Probably following the usage in the fields such as Transformation Groups, or  Transformation Geometry, the phrase ``an affine transformation" usually means a bijective affine map.  On the other hand, in Linear Algebra,  the phrase ``a linear transformation" is used for non-bijective linear maps as well.  To avoid confusion, and also
for brevity, we use a neutral terminology ``linear maps" or
``affine maps" for not necessarily bijective maps.

We may also like to  define 
\begin{equation}
(A_1, v_1) + (A_2, v_2) = (A_1 + A_2, v_1 + v_2).
\end{equation}
As is well-known, $L(\V)$ becomes an associative
$F$-algebra with this definition of addition, and taking
composition as multiplication. However, we note that with the same
definitions, in $A(\V)$, we do {\it not} get left distributivity of
multiplication w.r.t addition. So $A(\V)$ becomes only a ``near ring", 
or better a ``near $\F$-algebra",  cf. for example, \cite{pilz}. Let $GL(\V),$
resp. $GA(\V),$ denote the subsets of $L(\V)$, resp $A(\V)$
consisting of invertible elements. They form groups under
composition, and $GL(\V)$ is a subgroup of $GA(\V)$. They act on
$L(\V)$ resp. $A(\V)$ by conjugation. Namely $f$ in $GL(\V)$,
resp. $GA(\V)$, acts on $L(\V)$, resp. in $A(\V)$ by $T \mapsto
fTf^{-1}$. We denote these actions by $\phi_L$ resp $\phi_A$. When
there will be no confusion, we shall also abbreviate them to
$\phi$.

\smallskip
Our interest in this paper is to study the ``dynamics" of $L(\V)$
and $A(\V)$. We interpret the words ``study of dynamics" to mean

\smallskip

i) Parametrization of the $\phi$-orbits of $GL(\V)$, resp.
$GA(\V),$ on $L(\V)$, resp. $A(\V)$, cf. theorem 7.1.

\smallskip

ii) In any action of a group $G$ on a set $X$ we have a notion of
orbit-equivalence. Namely,  $x$, $y$ in $X$ are {\it
orbit-equivalent} iff the stabilizer subgroups $G_x$, and $G_y$
are conjugate,  cf. \cite{kulkarni}, theorem 2.1 for a precise statement on the structure of an orbit-equivalence class, as a certain set-theoretic fibration. In the case of the
$\phi$-action  a stabilizer subgroup at $T$ in $GL(\V)$ resp.
$GA(\V)$ is precisely the centralizer of $T$ in $GL(\V)$ resp.
$GA(\V)$. We denote this subgroup by $Z_L^*(T)$, resp. $Z_A^*(T)$.
For short, we call the orbit-equivalence in either the linear or
the affine case, the {\it $z$-equivalence}. In this paper one of
our main aims is to parametrize the $z$-equivalence classes of
linear or affine maps, cf. theorem 7.2.  

\smallskip

iii) Parametrizations of linear, resp. affine, maps which depend
only on $\F$ and $dim \, \V = n,$ and not on the choice of a
linear resp. affine coordinate system, cf. theorem 7.3.

\smallskip

Interestingly, in this case $GL(\V)$, resp. $GA(\V)$, are also
subsets of $L(\V),$ resp. $A(\V)$, so there is also a notion of
centraliers of $T$ is $L(\V),$ resp $A(\V)$. We denote these
centralizers by $Z_L(T)$, resp $Z_A(T)$. Then $Z_L(T)$ is an
$\F$-subalgebra of $L(\V)$, and $Z_A(T)$ is a sub-near-$\F$-algebra
of $A(\V)$. In fact, $Z_L^*(T)$, resp. $Z_A^*(T)$, are precisely the
groups of invertible elements in $Z_L(T)$, resp $Z_A(T)$.

A basic notion of ``equivalence of dynamics" in our case is the
following. First, let $T_i$ be elements of $L(\V_i)$, i = 1, 2. We
say that the $T_i$'s are ``dynamically equivalent" if there is a
linear isomorphism $h: \V_1 \rightarrow \V_2$ such that $h\circ
T_1 = T_2\circ h$. In this case we shall also say that the pairs
$(\V_i, T_i)$, i = 1,2, are dynamically equivalent. Similarly let
$T_i$ be elements of $A(\V_i)$, i = 1, 2. We say that the $T_i$'s
are ``dynamically equivalent" if there is an affine isomorphism
$h: \A_1 \rightarrow \A_2$ such that $h\circ T_1 = T_2\circ h$.

Next, let $T$ be an element of $L(\V)$. We say that $\V$ is {\it
$T$-decomposable}, or the pair $(\V, T)$ is {\it
decomposable}, or more loosely also that $T$ is {\it
decomposable}, if $\V$ is  a direct sum of proper $T$-invariant
subspaces. Otherwise $\V$ is said to be {\it $T$-indecomposable}.
Since $dim\, \V$ is finite, clearly $\V$ is a direct sum of
finitely many $T$-indecomposable invariant subspaces. Also  the
pair $(\V, T)$ is indecomposable iff any dynamically equivalent
pair is indecomposable. So from a dynamic viewpoint, a basic problem is to describe suitable
models of indecomposable $(\V, T)$'s, and secondly, given a pair $(\V, T)$ to understand in general all decompositions 
of $\V$ into $T$-indecomposable subspaces.
The first problem is solved by the theory of
``rational canonical form of  a square matrix" as a special case of modules over
principal ideal domains.  In this classical approach, the second problem gets obscured in a clever inductive proof. Following a dynamic viewpoint,
we shall offer a new view of both the  problems which,  
in some sense, is  ``dual'' to the the classical approach.

This approach strengthens the classical theory in a number of ways.
In this paper we have considered the following aspects.

i) Making an essential use of the $Z_L(T)$-invariant flag, we determine the conjugacy classes,   the centralizers, and $z$-classes of both linear and affine
maps. The consideration of affine maps  
naturally arises in the study of affine ODEs (where $\F = \R$), cf. section 6. The texts of ODEs appeal to a general ``method of variation of parameters".  In our opinion, the dynamic approach offers a better insight. At the same time, 
we are not aware of  any literature on the general case, from the 
viewpoint of $GA(\V)$-action on $A(\V).$

ii) We derive a  necessary and sufficient condition for the existence of ``S + N"- decomposition of an operator -- or its multiplicative analogue, the ``SU"-decomposition of an invertible operator -- and its relation to lifts of $\E$-structures, cf. section 5. A basic observation going back to Maurer in special cases, cf. \cite{borel},  \cite{borel1},  
\cite{chevalley}, is that a linear algebraic group over a field of characteristic 0, contains 
the semisimple and unipotent parts of each of its elements. If the base-field is of positive characteristic such decomposition in general does not exist. In this context one introduces the notion of perfectness of the base-field, which is a sufficient condition for the existence of such decomposition. The dynamic viewpoint provides an overall insight on this ticklishly confusing point in the theory of linear algebraic groups. 

iii) We derive a  dynamic interpretation of Frobenius'   ``double commutant"  theorem, cf, section 4. 

iv) We prove the following {\it finiteness} result: {\it If $\F$ has the property that there are only finitely many field-extensions of $\F$ of degrees at most $n,$   then there are 
only finitely many $z$-classes in $L(\V),$ and $A(\V).$} For example, if $\F$ is an algebraically closed field, a real closed field, or a local field then $\F$ has the stated  property. This is a major example which illustrates the viewpoint that motivated \cite{kulkarni}. In the forthcoming papers we hope to extend this work to transformations in other classical geometries.

v) We obtain the generating functions for $z$-classes of linear maps, in some cases when there are only finitely many such $z$-classes in each dimension. They are related to the generating function for partitions in an interesting way. They appear to be a new type of generating functions which have not appeared in number theory before. We only make some elementary observations regarding these generating functions. 

\bigskip 

Before closing this introduction, we would like to remark that from a dynamic viewpoint, the minimal polynomial
$m_T(x)$ is perhaps a more basic invariant than the more easily computable 
invariant $\chi_T(x),$ the characteristic polynomial of $T$. One indication of
this fact is that $m_T(x)$ is defined even when $\V$ is infinite-dimensional.
Most results of this paper can be suitably extended to the case when 
$\V$ is infinite-dimensional, and $m_T(x) \not= 0$. However, to keep a focus 
we do not elaborate on  this direction here, as we had done in \cite{kulkarni}.

I wish to acknowledge the benefit of many conversations on the contents of this
work with Rony Gouraige. His thesis, cf. \cite{gouraige},  (City University of New York, 2006) partially extends this work to operators on finite-dimensional vector spaces over skew-fields.
It is also a pleasure to acknowledge some conversations with  I. B. S. Passi, and 
Surya Ramana at the Harish-Chandra Research Institute, Allahabad, India, regarding the ``S + N"-decomposition.

\section{Classical Theory for $L(\V)$}

Let $m_T(x)$ denote the minimal polynomial of $T$.
If $\F[T]$ denotes the $\F$-algebra generated by $T$, then $F[T]
\approx F[x]/(m_T(x))$. Let $m_T(x) = \Pi_{i=1}^{r}p_i(x)^{d_i}$
be the decomposition into irreducible factors. Here $p_i(x)$ is a
monic irreducible polynomial in $\F[x]$, and $p_i(x)$'s are
pairwise distinct. We shall call $p_i(x)$'s the {\it primes}
associated to $T.$ The first step in the theory provides a
decomposition $\V = \oplus_{i=1}^{r}\V_i$ into $T$-invariant
subspaces. Here ${\V}_i = ker \; p_i(T)^{d_i}$. 
We observe that this decomposition is in fact
invariant under $Z_L(T)$, the $\F$-subalgebra of $L(\V)$ 
consisting of all operators commuting with $\T.$ 
Let $T_i$ denote the restriction of $T$
to $\V_i$. Then $m_{T_i}(x) = p_i(x)^{d_i}$. Moreover we have a
canonical $\F$-algebra decomposition. 
\begin{equation}
Z_L(T) = \Pi_{i=1}^rZ_L(T_i).
\end{equation} 
So to describe the indecomposable pairs
$(\V, T)$ we have reduced to the situation where $m_T(x) =
p(x)^d$, where $p(x)$ is a monic irreducible polynomial in
$\F[x]$.

At this point, we note a crucial example. Consider the algebra $\V
= \F[x]/(p(x)^d),$ but consider it only as an $\F$-vector space.
For $u(x)$ in $\F[x]$ let $[u(x)]$denote the class of $u(x)$ in
$\F[x]/(p(x)^d).$ Let $T = \mu_x$ be the operator $[u(x)] \mapsto
[xu(x)].$ For $i = 0, 1, \ldots, d$, let $\V_i = \{[f(x)p(x)^i] |
f(x) \in \F[x]\}.$ Clearly we have a flag of subspaces
$$0 = {\V}_d \subset {\V}_{d-1} \ldots \subset {\V}_1 \subset
{\V}_0 = {\V}.$$ The claim is that $\V_i$'s are precisely {\it
all} the $T$-invariant subspaces of $\V.$ Indeed let $\W$ be a
$T$-invariant subspace of $\V.$ If $[f(x)p(x)^i]$ is in $\W,$ then
by $T$-invariance, for all $g(x) $ in $\F[x]$, we have
$[g(x)f(x)p(x)^i]$ also in $\W.$ Let $i$ be the least non-negative
integer such that $\W$ contains an element of the form
$[f(x)p(x)^i]$ such that $p(x)$ does not divide $f(x).$ Then
$[f(x)]$ is a unit in the algebra $\F[x]/(p(x)^d).$ So $[p(x)^i]$
is in $\W.$ It easily follows that $\W = \V_i.$ Notice that no
$\V_i$ has a proper complementary subspace. So $(V, T)$ is an
indecomposable pair.

For the future, notice that in this case $dim_{\F} \; {\V} = deg
\;p(x)^d = d\; deg\, p(x)$.

A second and major step in the theory is that the converse of the
observation in the above example is true.

\medskip

\begin{theorem}  {Let $(\V, T)$ be an indecomposable pair. Then it is
dynamically equivalent to $(\F[x]/(p(x)^d), \mu_x),$ for some
monic irreducible polynomial $p(x)$ in $\F[x].$}
\end{theorem}

\smallskip

In view of the reduction in the first step, clearly an equivalent
statement is the following.

\medskip

\begin{theorem} {Let $(\V, T)$ be a  pair such that $m_T(x) = p(x)^d$,
where $p(x)$ is a monic irreducible polynomial $p(x)$ in $\F[x], $ of degree $m.$
Then $(\V, T)$ is   a direct sum of $T$-invariant 
indecomposable subspaces, each dynamically equivalent to
$(\F[x_i]/(p(x_i)^{d_i}), \mu_{x_i})$. Here $d_i \le d$, 
for at least one $i$, we have $d_i = d$, and $dim \; \V =  m \, \sum_i d_i.$}

\end{theorem}

\medskip

We note that in case $d = 1,$ the proof of either of these
statements is easier than in the classical approach dealing with
the more general situation of finitely generated modules over a
PID. Indeed observe that $\E = \F[x]/(p(x))$ w.r.t. to its
standard additive and multiplicative structures is {\it a field}.
In fact it is a simple field extension of $\F$. Here ``simple"
means that $\E$ is generated over $\F$ by a single element $[x]$.
Indeed $[x]$ is a root of $p(x)$ in $\E$, and in the language of
field theory $[x]$ is {\it a primitive element} of $\E$ over $\F$.
Thus the operation of $T$ on $\V$, which amounts to multiplication
by $[x]$, or $[x_i]$'s in the standard models $(\F[x_i]/(p(x_i)),
\mu_{x_i})$, equips $\V$ with the structure of a vector space over
$\E$, which extends its structure as a vector space over $\F$. In
this $\E$-structure, the $T$-invariant subspaces are precisely the
$\E$-subspaces of $\V$. Also an $\F$-linear operator $S$ is in
$Z_L(T)$ iff $S$ is an $\E$-linear operator. It follows that $(\V,
T)$ is indecomposable iff $dim_{\E}\,\V = 1$. Equivalently, $(\V,
T)$ is decomposable iff $dim_{\E} \V = r \ge 2$. In this case, a
choice of an $\E$-basis leads to a $T$-invariant decomposition of
$\V$ into $T$-indecomposable subspaces. The ambiguity in the
choice of a $T$-invariant decomposition of $\V$ is precisely the
ambiguity of choosing an $\E$-basis. Here $Z_L(T) \approx L_{\E}(\V),$
and $Z_L(T)^*  \, \approx \, GL_{\E}(\V).$ The orbits of $Z_L(T)^* $ 
on $\V$ are $\{0\}$, and $\V - \{0\}.$ As a module over the associative $\F$-algebra 
$Z_L(T)$ or the group $Z_L(T)^*,$  $\V$ is irreducible. Moreover the $T$-action is
{\it dynamically semi-simple} in the sense that every
$T$-invariant subspace has a $T$-invariant complement.

A word of caution regarding the use of the phrase ``dynamically semi-simple".
There is another notion of semi-simplicity: namely,  $T$ is {\it
algebraically semi-simple} if it is diagonalizable on $\V \otimes
{\tilde \F}$, where ${\tilde \F}$ denotes the algebraic closure of
$\F$, cf. \cite{borel}.   Contrary to some mis-statements in the
literature the notions of algebraic semi-simplicity and dynamic
semi-simplicity are {\it not} equivalent. They differ when the
characteristic of $\F$ is not $0$, and $\F$ is not perfect. See
section 5.

Now note that as an associative $\E$-algebra,
$Z_L(T)$  is  {\it simple}, and $\E$ can be recovered from
$Z_L(T)$ as its center. 
Since $\E$ is a simple field extension, we have also verified
Frobenius's well-known ``bi-commutant theorem", in the special case $d = 1$, namely
an operator which commutes with every operator which commutes with $T$ is 
a polynomial in $T.$

The case $d \ge 2$ is much more difficult. In this case the
dynamical approach provides a different, in some sense ``dual",
insight, over the classical theory. We turn to  this case in the
next section.

\section{Orbits of $Z_L(T)^*$, and a Canonical Maximal  $Z_L(T)$-invariant  Flag}

Let $T \in L(\V), m_T(x) = p(x)^d,$ where $p(x)$ is a monic
irreducible polynomial $p(x)$ in $\F[x]$. Let $deg\, p(x) = m$. So
$\E = \F[x]/(p(x))$ is a simple field extension of $F$, and
$dim_{\F}\, \E = m$. Assume $d \ge 2.$ Let $N = p(T),$ and $\V_i =
ker\, N^i, i = 0, 1, 2, \ldots d.$ Thus we have a
$Z_L(T)$-invariant flag of subspaces $$0 = \V_0 \subset \V_1
\subset \V_2 \subset \ldots \subset \V_d = \V.$$

We note an immediate consequence. Let ${\bar T}_i$ denote the
operator induced by $T$ on $\V_i/\V_{i-1}, i = 1, 2, \ldots d$.
Then $m_{{\bar T}_i}(x) = p(x)$. So by the case $d = 1$ treated in
the previous section we see that  $\V_i/\V_{i-1}$ has a canonical
$\E$-structure. So $dim_{\F}\V_i/\V_{i-1}$, and finally
$dim_{\F}\V$ is divisible by $m$. So let $n = dim\, \V = ml$.

We shall obtain a canonical, maximal $Z_L(T)$-invariant refinement
of this flag. It will be convenient to use a double-subscript
notation $\V_{i,j}$ for the subspaces occurring in this refined
flag, with the understanding that $\V_i = \V_{i,0}$.  
 If we insert $k-1$ new terms between $\V_i$ and
$\V_{i+1},$ we shall also denote $\V_{i+1}$ by $\V_{i, k}$. Our
basic observation is: $\V_{i,j} - \V_{i,j-1}$ are precisely the
$Z_L(T)^*$-orbits on $\V$. In particular, $\V_{i,j}/\V_{i,j-1}$
are irreducible, when they are considered as modules either over
the group $Z_L(T)^*$ or over the $\F$-algebra $Z_L(T)$. Let ${\bar
T}_{i,j}$ denote the operator induced by $T$ on $\V_{i,j} /
\V_{i,j-1}.$ It will turn out that $m_{{\bar T}_{i,j}}(x) = p(x)$.
So by the case $d = 1$ discussed in the last section, we have a
canonical $\E$-structure on $\V_{i,j} / \V_{i,j-1}.$ Let $\sigma =
dim_{\E} \V_{i,j} / \V_{i,j-1},$ and let $\W_{\sigma}$ denote an
(abstract) vector space of dimension $\sigma$ over $\E$. As it
will turn out, the algebra of operators induced by $Z_L(T)$ on
$\V_{i,j} / \V_{i,j-1}$ is dynamically equivalent to the standard
action of $L_{\E}(\W_{\sigma})$ on $\W_{\sigma}.$

Before running into the proofs of these assertions, for the
convenience of the reader, let us reconcile, albeit partially,
this description with the classical theory. The classical theory
attaches to $T$ as above, its {\it elementary divisors}, which are
polynomials of the form $p(x)^{s_i}, i = 1, 2, \ldots r$.   We may
assume that $1\le s_1 < s_2 < \ldots < s_r = d$ are the distinct
exponents of these elementary divisors, and   $\sigma_i$ is the
multiplicity of $p(x)^{s_i}$. Then $l = \Sigma_{i=1}^{r}
s_i\sigma_i$, where $n = dim\, \V = m\,l, m = deg\, p(x)$.
According to the classical theory the pair $(\V, T)$ is
dynamically equivalent to the direct sum of the pairs of the form
$(\F[x]/(p(x)^{s}, \mu_x)$ where $s = s_i$ occurs $\sigma_i$
times, $i = 1, 2, \ldots, r$. It will turn out that the dimensions
$\sigma$ of the (abstract) $\E$-vector spaces $\W_{\sigma}$
mentioned in the previous paragraph are precisely the
multiplicities $\sigma_i$'s of the elementary divisors in the
classical theory. The refined flag mentioned above will
independently pick up the exponents $s_i$'s and
multiplicities $\sigma_i'$s of the elementary divisors, subject to
the relations, $l = \Sigma_{i=1}^{r} s_i\sigma_i$, where $n =
dim\, \V = m\,l, m = deg\, p(x)$.

Let us now start building the refined flag. We shall first
describe the refined flag where the dimensions of the subspaces in
the flag  are non-decreasing, and then offer a second description
where these  dimensions   are strictly increasing.

\medskip

\begin{lemma} i) For $i > 0$, $N = p(T)$ maps $\V_i$ into $\V_{i-1}$, and

ii) For $i > 1$ the  map induced by $N$ on $\V_i/\V_{i-1}
\rightarrow \V_{i-1}/\V_{i-2}$ is injective.

\end{lemma}
\medskip

The proof is straightforward, and is omitted.

\smallskip

Let $(e_1, e_2, \ldots e_k)$ be elements in $\V_d $ whose images
$({\bar e}_1, {\bar e}_2, \ldots {\bar e}_k)$ form an $\E$-basis
of $\V_d / \V_{d-1}$. Then $T^j(e_i), 1\le j \le m-1, 1 \le i \le
k$ are   linearly independent over $\F$, as they are indepenedent 
over $\F$ mod $\V_{d-1}$. Let $\W_d$ denote
the $\F$-span of $T^j(e_i)$. Notice that by construction, $\V =
\V_d = \V_{d-1} + \W_d$ is a direct sum of subspaces. Among these,
$\V_{d-1}$ is $T$-invariant, but $\W_d$ is not (since we have
assumed $d \ge 2$). However by construction, mod $\V_{d-1}$ it is
$T$-invariant. We shall call such subspace of $\V_d$ {\it an
almost $T$-invariant} subspace. Now notice that $N$ maps $\W_d$
injectively in $\V_{d-1}$ as a subspace complementary to
$\V_{d-2}$. Moreover it is easy to check that $\V_{d-2} + N(\W_d)$
is independent of the choice of $\W_d$. It is a $T$-invariant, in
fact $Z_L(T)$-invariant, subspace of $\V_{d-1}$. In case $\V_{d-2}
+ N(\W_d)$ is a proper subspace of $\V_{d-1}$ we insert it as an
additional subspace in the flag between $\V_{d-2}$ and $\V_{d-1}.$
Notice that $(\V_{d-2} + N(\W_d))/\V_{d-2}$ is an $\E$-subspace of
$\V_{d-1}/\V_{d-2}$.

\smallskip

Assume that $\V_{d-2} + N(\W_d)$ is a proper subspace of
$\V_{d-1}$. For convenience, denote $k = dim_{\E} \V_d / \V_{d-1}$
by $k_d$, and $e_i$ by $e_{d, i}$. Let $k_{d-1} = dim_{\E}
\V_{d-1}/\V_{d-2}\, - \, dim_{\E} (\V_{d-2} + N(\W_d))/\V_{d-2}$.
If $k_{d-1} \not= 0$, choose $e_{d-1, i}, 1\le i \le k_{d-1}$ in $\V_{d-1}$ so that
their classes mod $\V_{d-2}$ form an $\E$-basis of a subspace of
$\V_{d-1}/\V_{d-2}$ complementary to $(\V_{d-2} +
N(\W_d))/\V_{d-2}$. Then $T^j(e_{d-1, i}), 1\le j \le m-1, 1 \le i
\le k_{d-1}$ are clearly linearly independent over $\F$. Let
$\W_{d-1}$ denote the $\F$-span of $T^j(e_{d-1, i})$'s. Then
$\W_{d-1}$ is an almost $T$-invariant subspace of $\V_{d-1}$. Then
$N$ maps $\W_{d-1}$ injectively into $\V_{d-2}$ onto a subspace
complementary to $\V_{d-3} + N^2(\W_d)$. If $\V_{d-3} + N^2(\W_d)
+ N(\W_{d-1})$ is a proper subspace of $\V_{d-2}$ we insert it as
an additional subspace in the flag between $\V_{d-3} + N^2(\W_d)$
and $\V_{d-2}$. We note again that two subspaces $\V_{d-3} +
N^2(\W_d)$ and $\V_{d-3} + N^2(\W_d) + N(\W_{d-1})$ are
independent of the choices of $\W_d$ and $\W_{d-1}$, and they are
$Z_L(T)$-invariant subspaces of $\V_{d-2}$. In case $\V_{d-2} + N(\W_d)$ 
is not a proper subspace of
$\V_{d-1}$, we simply take $\W_{d-1}$ to be 0, and continue.

Proceeding in this way we obtain the following refined flag, where
the dimension of the subspaces are non-decreasing.

$$0 = \V_0 \subset N^{d-1}(\W_{d}) \subset N^{d-1}(\W_{d}) +
N^{d-2}(\W_{d-1}) \subset \ldots $$

$$N^{d-1}(\W_{d}) +
N^{d-2}(\W_{d-1}) + \ldots N(\W_2) + \W_1 = \V_1 \subset$$

$$ \V_1 + N^{d-2}(\W_{d}) \subset \V_1 +
N^{d-2}(\W_{d}) + N^{d-3}(\W_{d-1}) \subset \ldots $$

$$\V_1 + N^{d-2}(\W_{d}) + N^{d-3}(\W_{d-1}) + \ldots N(\W_3) + \W_2 =
\V_2 \subset \ldots $$

$$\ldots \ldots \ldots$$

$$\V_{d-3}\subset \V_{d-3} + N^2(\W_d)\subset \V_{d-3} + N^2(\W_d) + N(\W_{d-1})\subset$$

$$ \V_{d-3} + N^2(\W_d) + N(\W_{d-1}) + \W_{d-2} = \V_{d-2}\subset$$

$$ \V_{d-2} + N(\W_d)\subset
\V_{d-2} + N(\W_d)+ \W_{d-1} = \V_{d-1} \subset \V_{d-1} + \W_d = \V_d.$$

\smallskip

Notice that in this flag the sum $\oplus_{j=0}^{d-1} N^j(\W_d)$
forms a $T$-invariant (but {\it not} $Z_L(T)$-invariant) subspace
dynamically equivalent to $k_d$ copies of $\F[x]/(p(x)^{d})$. More
generally the sums $\oplus_{j=0}^{s-1} N^j(\W_s), s = 1, 2,
\ldots ,d$ form a $T$-invariant (but {\it not} $Z_L(T)$-invariant)
subspace dynamically equivalent to $k_s$ copies of
$\F[x]/(p(x)^{s})$, where $mk_s = dim \,\W_s$. If $\W_s = 0$, then
those terms effectively do not occur. By construction,  $\W_s$ is
the $\F$-span of $T^je_{s, 1}, \ldots T^je_{s, k_s}, 0\le j \le
m-1.$ So $N^uT^je_{s, 1}, \ldots N^uT^je_{s, k_s}, 0\le j \le m-1,
0\le u \le s-1$ is a basis of $\oplus_{j=0}^{s-1} N^j(\W_s)$.

\smallskip

To get an irredundant flag where the dimesions are strictly
increasing we need to proceed as follows. Let $1\le s_1 < s_2 <
\ldots <s_r = d$ be integers such that $\W_{s_i} \not= 0$. Let
$m\sigma_i = dim \, W_{s_i}, 1\le i \le r$. Let $\V_i = \V_{i, 0}$,
and for $0 \le i \le s_{r-j+1}, \; 1 \le j \le r$, set $$\V_{i, j}
= \V_i + N^{s_r - i}(\W_{s_r}) + N^{s_{r-1} - i}(\W_{s_{r-1}}) + \ldots + 
N^{s_{r-j+1} - i}(\W_{s_{r-j+1}}).$$

\smallskip

A final important observation deals with the ambiguities involved
in the choices of $\W_{s}$ where $s$ is one of the $s_i'$s. Notice
that by construction,  $\W_s$ is the $\F$-span of $T^je_{s, 1},
\ldots T^je_{s, k_s}$,  $0\le j \le m-1$. So $N^uT^je_{s, 1}, \ldots
N^uT^je_{s, k_s}, 0\le j \le m-1, 0\le u \le s-1$ is a basis of
$\oplus_{j=0}^{s-1} N^j(\W_s)$. Let $\W'_{s}$ be another choice of
almost T-invariant subspace complementary to the subspace previous
to $\V_{s+1}$ in the refined flag. Suppose $\W'_{s}$ is constructed starting with
$e'_{s, 1}, e'_{s, 2}, \ldots, e'_{s,k_s}$.
Let $N^uT^je'_{s, 1}, \ldots
N^uT^je'_{s, k_s}, 0\le j \le m-1, 0\le u \le s-1$ be the
corresponding basis of $\oplus_{j=0}^{s-1} N^j(\W'_s)$. Then
consider the $\F$-linear map which sends $N^uT^je_{s, v}$ to
$N^uT^je'_{s, v}$ and which is identity on the remaining
$\oplus_{j=0}^{t-1} N^j(\W_t)$ for $t\not= s$. Clearly this map
is invertible, commutes with $T$, and carries $\W_s$ into $\W'_s$.
In particular, repeating this argument to any two successive terms
$\V_{i, j}$ and $\V_{i, j + 1}$ we see that $Z_L(T)^*$ is
transitive on $\V_{i, j + 1} - \V_{i, j}$. In particular, this   implies
that $\V_{i, j + 1}/\V_{i, j}$ is irreducible as a module over the
group $Z_L(T)^*$ or the associative algebra $Z_L(T)$.

\smallskip

In particular, this completes the proof of theorem (2.1), or as noted earlier, 
equivalently, of theorem (2, 2). In the process however, we have strengthened the result 
which we record in the following form.  

\smallskip

\begin{theorem}{Let $T$ be in $\L(\V)$, $m_T(x) = p(x)^d$, where $p(x)$ is a monic irreducible polynomial in $\F[x]$.  Then   $\V$ admits a canonical, maximal $Z_L(T)$-invariant flag. A complement of each term appearing in the flag in its succeeding term is an orbit of $Z_L(T)^*$.  In particular the quotient of each term appearing in the flag by its preceding term is an irreducible module over the group $Z_L(T)^*$, or the $\F$-algebra 
$Z_L(T).$}  
\end{theorem}

\medskip

As a by-product of this proof, we have some interesting
dimension-counts,  which refine the dimension counts in the well known Frobenuis' dimension formula. For simplicity, let $f(x) = p(x)^d$, where $p(x)$ is a monic irreducible polynomial in $\F[x]$, $n = dim\, \V$, $m = deg\, p(x),$ and $n = m\,l.$ Consider a partition of $l$, namely, $l = \sum_{i=1}^{r} s_i \sigma_i,$ where $s_i$ occurs $\sigma_i$ times, and we have assumed $1\le s_1 < s_2 <
\ldots <s_r = d.$ This data uniquely determines a pair $(\V, T)$ up to dynamical equivalence, with $m_T(x) = f(x).$

 To start with, we note again

\smallskip

$\bullet$ $n = dim\,\V = m (s_r\sigma_r + s_{r-1}\sigma_{r-1}
+\ldots + s_1\sigma_1.)$

\smallskip

$\bullet$ The successive sub-quotients associated to the flag,
starting from $\V_0 = 0$, which are the Jordan constituents of
$\V$ considered as a module over the associative algebra $Z_L(T),$
have $\F$-dimensions

$(m \sigma_r, m \sigma_{r-1}, \ldots ,m\sigma_1)$ occurring $s_1$
times,

$(m \sigma_r, m \sigma_{r-1}, \ldots , m \sigma_2)$, occurring $s_2-s_1$
times,

$\ldots$

$ (m \sigma_r, m \sigma_{r-1})$ occurring $s_{r-1} - s_{r-2}$ times,

$(m \sigma_r)$ occurring $s_r - s_{r-1}$ times.

\smallskip

$\bullet$ Let $\tau_i = \sigma_r + \sigma_{r-1}+ \ldots +\sigma_i,
1\le i \le r$. Then from the refined flag we see that $dim\, im \, N = n -
m\tau_1,$ and so $dim\, ker N = m\tau_1$.

\smallskip

As an associative $\F$-algebra $R = Z_L(T)$ has its nil-radical
$nil\, R$, and $R/nil\, R$ is a semi-simple $\F$-algebra. The
elements of $R$ which map $\V_{i, j +1}$ into $\V_{i, j}$ clearly
form a nilpotent ideal $I$ of R. On the other hand,  $R/I$ is
clearly isomorphic to a direct product of $L(\W_{s_i})$, $i = 1, 2, \ldots r$. So 
$I$ is   $nil \; R.$
This is worth recording as a theorem.

\begin{theorem}{Let $T$ be in $L(\V)$, with $m_T(x) = p(x)^d$ where $p(x)$ is a monic irreducible polynomial in $\F[x]$. Let $\E = \F[x]/(p(x))$. Then $R = Z_L(T)$ considered as an associative algebra
has its maximal semisimple quotient isomorphic to a direct sum of matrix rings 
$M_{\sigma_i}(\E),$ where $\sigma_i$ are as defined above.}
\end{theorem}

\smallskip

In particular,

$\bullet$ $dim\, R/nil\, R = m (\sigma_1^2 + \sigma_2^2 + \ldots
\sigma_r^2).$

\smallskip

\section{Strongly Commuting Operators}

Let $T$ be in $L(\V).$ We say that an operator $S$ in $L(\V)$  
{\it strongly commutes} with $T$ if $S$  commutes with $T$, and  leaves  invariant every
$T$-invariant subspace of $\V.$ It is interesting to compare the
following theorem with Frobenius' bicommutant theorem. It will be useful later on.

\smallskip

\begin{theorem} {Let $T$ be in $L(\V).$ An operator $S$ in $Z_L(T)$
strongly commutes with $T$ iff $S$ is in $\F[T].$}
\end{theorem}

\smallskip

\begin{proof} The ``if" part is clear. Conversely, suppose that $S$
in $Z_L(T)$   strongly commutes with $T.$

First consider the case when $(\V, T)$ is dynamically equivalent
to $(\F[x]/(p(x)^d), \mu_x),$ where $p(x)$ is a monic irreducible
polynomial in $\F[x].$ Let $S$ be in $Z_L(T),$ and $S(1) =
[f(x)].$ It is easy to see that $S = f(T).$  Thus $Z_L(T) =
\F[T],$   and so every element in $Z_L(T)$ strongly commutes with
$T.$

Next consider the case where $m_T(x) = p(x)^d,$ and $p(x)$ is a
monic irreducible polynomial in $\F[x].$ Then $\V$ is a direct sum
of $T$-invariant subspaces $\W_i$,  dynamically equivalent to
$(\F[x_i]/(p(x_i)^{d_i}), \mu_{x_i})$, $1\le i\le k$, and $d= d_1
\ge d_2 \ge \ldots  \ge d_k.$ Let $e_i, 1\le i \le k$ be a
$T$-module generator in $\W_i$.

Let $S|_{\W_1} = q_1(T)$ where $q_1(x)$ is a unique polynomial of
degree at most $d\,m, m = deg\, p(x).$ For $j \ge 2$ let $q_j(x)$ be
the polynomial of degree at most $d_j\,m,$ such that $S|_{\W_j} =
q_j(T)e_j.$ Then $S(e_1 + e_2) = q_1(T)e_1 +  q_j(T)e_j.$
On the other hand, since $S$ strongly commutes with
$T$, we also have $S(e_1 + e_j) = u(T)(e_1 + e_j)$ for some
polynomial $u(x)$ of degree at most $dm.$ It follows that $(q_1(T)
- u(T))e_1 = -(q_j(T)- u(T))e_j.$ Since $\W_1\cap \W_j = 0$ we
must have $(q_1(T) - u(T)) \equiv (q_j(T) - u(T)) \equiv 0 (mod \,
p(x)^{d_j}.)$ So $q_1(T) \equiv q_j(T) (mod \, p(x)^{d_j}.)$

Finally consider the general case. Write $m_T(x) =
\Pi_{i=1}^{r}p_i(x)^{d_i},$ where $p_i(x)$'s  are monic
irreducible polynomials in $\F[x].$ Let $\V = \oplus \V_i,$ where
$\V_i = ker \,p_i(x)^{d_i}$ be the corresponding primary
decomposition of $\V.$ Now $S$ leaves each $\V_i$  
invariant. We have shown $S|_{\V_i} = q_i(T)$ where $q_i(x)$ is a
uniquely determined polynomial mod $p_i(x)^{d_i}.$ By Chinese
Remainder Theorem, there exists a uniquely determined polynomial
$q(x)$ mod $m_T(x)$ which is congruent to $q_i(x)$ mod
$p_i(x)^{d_i}.$ This completes the proof.

\end{proof}

\section{Lifting  $T$-invariant $\E$-structures, and ``S+N"-decomposition}

Let $T$ be in $L(\V),$ and $\E$ an extension field of $\F$. An
$\E$-structure on $\V$ is an $\F$-algebra homomorphism
$\sigma_{\E} : \E \rightarrow L(\V).$ Such a homomorphism is
necessarily injective, and allows one to consider $\V$ as a vector
space over $\E$, {\it lifting} the structure of $\V$ as a vector
space over $\F.$ An $\E$-structure $\sigma_{\E}$ is said to be
$T$-invariant if the image of $\sigma_{\E}$ lies in $Z_L(T).$

\smallskip

Suppose that  $m_T(x) = p(x)^d,$ where $p(x)$ is a monic
irreducible polynomial in $\F[x]$. Let $\E = \F[x]/(p(x)).$ An
interesting problem is to investigate when $\V$ admits a
$T$-invariant $\E$-structure. When $d=1, \, T$ itself induces a
canonical $\E-$structure. Namely, $\F[T] \approx \E$, and the
inclusion mapping of $\F[T]$ in $Z_L[T]$ is a $T$-invariant
$\E$-structure. Assume $d \ge 2.$ Then $\V_i/\V_{i-1}$ admits a
canonical  $T$-invariant $\E$-structure, since the minimal
polynomial of the operator induced by $T$ on $\V_i/\V_{i-1}$ is
$p(x).$ Our concern is whether these canonical $\E$-structures
on $\V_{i}/\V_{i-1}$'s can be lifted to a canonical $\E$-structure on $\V$ 
itself.  By a ``canonical $\E$-structure on $\V$''  we mean:

i) Each $T$-invariant subspace is an $\E$-subspace.

ii) For each $i = 1, 2, \ldots , d,$ the induced $\E$-structure  on $\V_i/\V_{i-1}$
coincides with the one induced by $T$.

\smallskip

It will eventually turn out that if $\V$ admits an $\E$-structure which satisfies ii) then it also satisfies i). 
A first basic result in this direction 
is the following. For its importance in the theory of algebraic groups see below. In the following, for $f(x)$ in $\F[x],$ let $f'(x)$ denote its formal derivative. 

\smallskip

\begin{theorem} {Let $T$ be in $L(\V),$ $m_T(x) = p(x)^d,$ where $p(x)$ is
a monic irreducible polynomial in $\F[x]$, and  $\E =
\F[x]/(p(x)).$ Then $\V$ admits a $T$-invariant $\E$-structure iff
either $d = 1$ or $p'(x)$  is
not identically zero. Such structure is unique if it is canonical in
the sense that it satisfies i) and ii) stated above.}
\end{theorem}

\smallskip

\begin{proof} First we consider the issue of the existence of 
a $T$-invariant $\E$-structure.
We may assume $d \ge 2.$ If $deg\, p(x) = 1,$ we have
$\E = \F,$ and $p(x) = x - \alpha$ for some $\alpha$ in $\F.$ On
each $\V_{i+1}/\V_i, \, T$ acts as $\mu_{\alpha}: v \mapsto \alpha
v.$ Then clearly ${\tilde \mu_{\alpha}}$ defined by the same
formula acting on $\V$ is a $T$-invariant
$\E$-structure on $\V.$  Note that we have also $p'(x) \equiv 1 \not= 0,$
and the structure is canonical. So suppose $deg \, p(x) = m \ge 2.$

First  suppose that $p'(x)$ is not identically $0.$ Then $deg \,
p'(x) \le m -1.$ So $p(x), p'(x)$ are relatively prime.

Since $(\V, T)$ is dynamically equivalent to a direct sum of pairs
of the form $(\F[x]/(p(x)^e, \mu_{x})$ where $e \le d,$ and
$ \mu_x([u(x)]) \mapsto ([xu(x)]),$ it suffices to prove the
existence of $\mu_x$-invariant $\E$-structure in the special case
of $(\F[x]/(p(x)^e), \mu_{x}), e \ge 2.$ In this case $Z_L(\mu_x)
\approx \F[x]/(p(x)^e).$ For any $y \in \F[x]$ let $[y]$ denote
its class in $\F[x]/(p(x)^e).$ So the assertion of existence of an
$\mu_x$-invariant $\E$-structure amounts to the existence of a
polynomial $z = u(x) \in \F[x]$ such that the corresponding
operator $\mu_z$ has minimal polynomial  $ p(x).$

Since $p(x), p'(x)$ are relatively prime, there exist $a(x), b(x)
\in \F[x]$ such that $a(x)p(x) + b(x)p'(x) = 1.$ Consider $y = x -
b(x)p(x).$ (Notice that for any polynomial $u(x)$ in $\F[x]$,
$\mu_{u(x)p(x)}$ is nilpotent, and its minimal polynomial is of
the form $x^r, r \le d.$) Writing $\epsilon = - b(x)p(x),$ the
formal Taylor's theorem (for polynomials with coefficients in
commutative rings), gives

$$p(y) = p(x + \epsilon) = p(x) + \epsilon p'(x) +
\frac{\epsilon^2}{2} p''(x) + \ldots $$

$$\equiv p(x) (1 - b(x)p'(x)) +\ldots \equiv p(x)(a(x)p(x)) +\ldots
\equiv 0 \, (mod \, p(x)^2).$$ So $p(y)^r = 0,$ for a suitable $r <
e.$ It follows that $\mu_y$ has minimal polynomial of the form
$p(x)^r$ where $r < e.$ So $\F[[y]] \approx \F[x]/(p(x)^r), $ and
$\F[[y]] \subset \F[[x]].$ By induction on $e$ it follows that
there exists a polynomial $z = u(x) \in \F[x]$ such that the
corresponding operator $\mu_z$ has minimal polynomial $ p(x).$

\medskip

To prove the converse suppose that we have a pair $(\V, T), m_T(x)
= p(x)^d, d \ge 2$ where $p(x)$ is a monic irreducible polynomial
in $\F[x]$, $\E = \F[x]/(p(x)),$ and $\V$ admits a $T$-invariant
$\E$-structure. This implies the existence of $S$ in $Z_L(T)$ with
$m_S(x) = p(x).$   In the associated flag   $\V_2$ is
$S$-invariant.   We only need
to prove that $p'(x) \not\equiv 0.$ So we readily reduce to the case $d = 2.$
To arrive at a contradiction,
suppose that $p'(x)\equiv 0.$  Now notice that for any polynomial
$u(x)$ in $\F[x]$ we have by the formal Taylor's theorem,

$$p(S + u(T)) = p(S) + u(T)p'(S) + \ldots = p(S) = 0.$$
But then

$$p(T) = p(S + T - S) = p(S+T) -S p'(S+T) + \ldots = p(S + T) = 0.$$

This is a contradiction since we have assumed $m_T(x) = p(x)^2.$
So we must have $p'(x) \not\equiv 0.$

\smallskip

Next we consider the issue of uniqueness of a canonical $T$-invariant 
$\E$-structure. Let $\sigma_1:\E \rightarrow Z_L(T), \, \sigma_2:\E \rightarrow Z_L(T),$ 
be two canonical $T$-invariant $\E$-structures. By passing to a $T$-invariant subspace,
we may reduce to the case when $(\V, T)$ is 
dynamically equivalent to $\F[x]/(p(x)^d), \mu_x).$ Then $Z_L(T) = \F(T).$ 
Let $\alpha$ be a primitive element of $\E$ over $\F,$ and 
$\sigma_i(\alpha) = S_i, \, i = 1, 2.$ Let $S_i = f_i(T)$ where $f_i(x) \in \F[x]$ are 
well-defined polynomials mod $p(x)^d.$ Since $S_i$'s define canonical 
$T$-invariant $\E$-structures we see that we must have $f_i(x) = x$ mod $p(x).$
In the existence proof we considered polynomials $a(x), \, b(x)$ satisfying
$a(x)p(x) + b(x)p'(x) = 1.$ Notice that $b(x)$ is uniquely defined by the condition that 
$deg \, b(x) < m.$ This is also the unique choice so that $p(x - b(x)p(x)) \equiv 0$ mod 
$p(x)^2.$ The same argument shows that the induction procedure used in the existence proof
leads to a polynomial $f(x)$ uniquely determined mod $p(x)^d,$ such that $m_{f(T)}(x) = p(x).$
So $f(x) = f_1(x) = f_2(x),$ and hence $\sigma_1 = \sigma_2.$ This finishes the proof.

\end{proof}

\medskip

\begin{remark} Notice that the condition $p'(x) \not\equiv 0$ is
automatically satisfied if the characteristic of $\F$ is 0.
Suppose that the characteristic of $\F$ is $l > 0.$ Let $p(x) =
\sum_{i=0}^{m} a_ix^i,$  Then $p'(x)\equiv 0$ iff $a_i = 0$ unless
$i$ is a multiple of $l.$ So if $p'(x) \equiv 0$, then we may take $p(x) = \sum_{i=0}^{m'}
b_ix^{il},$ where $m = m'l$. When $l > 0,$ $\F$ is said to be {\it
perfect} if $u \mapsto u^l$ is an isomorphism. For example a
finite field, or an algebraically closed filed is automatically
perfect. Notice that the condition $p'(x) \not\equiv 0$ is
automatically satisfied if   $\F$ is perfect. For otherwise, we
can write $b_i = c_i^l$ and so $p(x) = {(\sum_{i=0}^{m'} c_i
x^i)}^l,$ which will contradict that $p(x)$ is irreducible over
$\F.$

\end{remark}

\medskip

\begin{definition}{Let $T$ be in $L(\V).$ A
``S+N"-decomposition   of $T$ is
a pair $S, N$  such that i) $T = S + N$, ii) $S$ is
dynamically semi-simple,  iii) $N$ is nilpotent, and iv) $SN = NS.$}
\end{definition}

\begin{remark} Usually this notion is defined where dynamic semisimplicity is 
replaced by a stronger condition of algebraic semisimplicity, 
cf. the remarks in the introduction. This notion is
basic in the theory of algebraic groups, cf. \cite{borel}, \cite{chevalley}, cf. also \cite{borel1} for historical remarks.
\end{remark}

\smallskip

\begin{theorem} {Let $T$ be in $L(\V),$ and $m_T(x) =
\Pi_{i=1}^{r}p_i(x)^{d_i},$ where $p_i(x)$'s  are monic
irreducible polynomials in $\F[x].$ Then 

1) $T$ admits a
``S+N"-decomposition iff for each $i$, either $d_i = 1$ or else
$p_i'(x) \not\equiv 0.$ 

2) If it exists, a ``S+N"-decomposition is unique.

3)  If  $m_T(x) = p(x)^{d}$, $p(x)$  is a monic
irreducible polynomial, $\E = \F[x]/(p(x)),$ and  a ``S+N"-decomposition exists, then $S$ defines 
the canonical $T$-invariant $\E$-structure on $\V$. In particular $S$ strongly commutes with $T$, and so 
$S$, and hence $N$, are polynomials in $T$.}
\end{theorem}

\begin{proof} Notice that
by the condition iv) in the definition of ``S+N"-decomposition, 
$S, N$ are in $Z_L(T).$ So they leave the $T$-primary 
decomposition of $\V$ invariant, and
their restriction to a $T$-primary component of $\V$ are the
semisimple and nilpotent components of the restriction of $T.$
So to investigate the existence of ``S+N"-decomposition we reduce to
the case where $m_T(x) = p(x)^d$, where $p(x)$ is a monic
irreducible polynomial in $\F[x].$ 

First consider the existence issue. If $d = 1,$ then $T$ is semisimple,
and taking $S = T,$ and $N = 0,$ we obtain a ``S+N"-decomposition.
So consider $d\ge 2.$ Let $\E = \F[x]/(p(x)).$ First suppose that 
$p'(x) \not\equiv 0.$ In the previous theorem
we observed that under this condition there exists a polynomial 
$f(x)$ in $\F[x]$   such that $S = f(T)$ defines a canonical $T$-invariant $\E$-structure on 
$\V.$ In particular $m_S(x) = p(x),$ and so $S$ is dynamically semisimple. 
Let $\bar T_i, \bar S_i,$ be the operators 
induced by $T, S$ respectively on $\V_i = ker \, p(T)^i, \, i =0,  1, 2, \ldots, d.$
Since $S$ defines a canonical $T$-invariant $\E$-structure  we have $\bar T_i = \bar S_i.$
It follows that $N = T - S$ is nilpotent, and hence $T = S + N$ is an 
``S+N"-decomposition   of $T. $

Conversely suppose $T = S + N$ is an ``S+N"-decomposition   of $T. $ Then 
the induced operatots $\bar T_i, \bar S_i,$ on $\V_i, i =   1, 2, \ldots, d,$
are commuting dynamically semisimple operators. So their nilpotent difference 
$\bar N_i$ must be $0.$ (See equation (5.1) below). 
So $m_{\bar T_i}(x) = p(x) = m_{\bar S_i}(x).$   Since $S$ is dynamically semi-simple, 
it follows that $m_S(x) = p(x)$ also. Let $\E = \F[x]/(p(x)).$ 
Thus $\F[S] \approx \E$, and $S$ defines a
$T$-invariant $\E$-structure on $\V.$ So $p_i'(x) \not\equiv 0.$

Now consider the issue of uniqueness of ``S+N"-decomposition. Again we reduce to the case 
when $m_T(x) = p(x)^d.$ Let $d = 1.$ Then $S = T, N = 0$ is one ``S+N"-decomposition.
Suppose $T = S + N$ any ``S+N"-decomposition. We need to show that $N = 0.$ Indeed,
\begin{equation}
{p(T) = p(S + N) = p(S) + Np'(S) + {\frac{N^2}{2!}}p''(S) + \ldots.}
\end{equation}
Notice that $p(T) = p(S) = 0,$ and $p'(S)$ is invertible. If $N \not=0$ then the rank of $N$
is greater than the rank of $N^i$ for $i \ge 2.$ So the above equation is not possible unless $N = 0.$ 
Now consider the case $d \ge 2.$  The proof is by induction on $d.$ 
Let $T = S + N, T = S_1 + N_1$ be two
``S+N"-decomposition. By induction we may assume $ S = S_1$ on $\V_{d-1},$ 
and $S, S_1$ induce the same operators on $\V_d/\V_{d-1}.$ It follows that we must have 
$S_1 = S + M$ where $M$ maps $\V_d$ into $\V_{d-1},$ and $V_{d-1}$ onto $0.$ 
Such $M$ must be nilpotent. 
\begin{equation}
{p(S_1) = p(S + M) = p(S) + Mp'(S) + {\frac{M^2}{2!}}p''(S) + \ldots.}
\end{equation}
By the same argument as above we see that $M = 0.$ It follows
that ``S+N"-decomposition, if it exists, is unique.

By uniqueness of ``S+N"-decomposition it follows that 
when $m_T(x) = p(x)^d$, and ``S+N"-decomposition exists,
then  $S$ defines the  canonical $T$-invariant $\E$-structure. 
So $S$ strongly commutes with $T$, and hence it (and so also $N$) is a polynomial in $T.$

\end{proof}

\smallskip

\begin{remark}The theorem 5.4 shows that in the definition of canonical $T$-invariant $\E$-structure
the condition i) is a consequence of condition  ii). On the other hand, to get uniqueness
of a $T$-invariant $\E$-structure it is clearly necessary to impose a condition 
such as ii). For example, if 
$m_T(x)$ and $m_{S}(x)$ are both irreducible, such that $\E = \F[x]/(m_T(x))
\approx \F[x]/(m_S(x))$ then $T$ and $S$ would usually define different $\E$-structures.
\end{remark} 

\smallskip

\begin{remark}As remarked earlier, perfectness of  $\F$ is a sufficient condition
for the existence of ``S+N"-decomposition. However the author is
not aware of a statement of a   necessary and sufficient
condition for the existence of ``S+N"-decomposition in the
literature. One may avoid the issue by defining a different notion
of semisimplicity, namely  ``algebraic semi-simplicity" to mean
that the operator is diagonalizable over the algebraic closure of
$\F.$ However, in author's opinion, it is desirable to have all
elements of the orthogonal group with respect to an anisotropic
quadratic form to be ``semisimple". This would not be the case if
we take the algebraic notion of semisimplicity. We also note that
among the fields of positive characteristic, an important class of fields, namely, 
function fields of algebraic varieties of positive dimensions over
finite fields, are not perfect. Lastly we note that there are
misleading remarks in the literature that the dynamic and
algebraic notions of semisimplicity are equivalent.

\end{remark}

\smallskip

\begin{remark}Finally we would like to remark on a  forgotten rational canonical form 
for matrices due to Wedderburn, \cite{wed}. In the
standard texts on Algebra, such as \cite{jacobson}, the authors present a
matrix for an operator by choosing a suitable basis, 
called its rational canonical form.  Let $T$ be
in $L(\V),$ and $m_T(x) = p(x)^{d},$ where $p(x)$  is a  monic
irreducible polynomial in $\F[x].$ For some authors the matrix presented to
represent $T$ involves a companion matrix of $p(x)^d.$ This is obviously a poor choice when 
$d \ge 2.$ It is better to use only the companion matrix of $p(x).$ It is worth
noting that one may use any matrix conjugate to a companion
matrix. (This remark is important even in the basic case when $\F
= \R,$ the field of real numbers, and important for the solutions
of linear first order ODEs with constant coefficients, cf. section 6.) But
secondly when $deg\, p(x) = m \ge 2$, and $d \ge 2$, the matrix presented is  
written in the form ``S+N" where $S$ (the matrix of diagonal
blocks) is semisimple and $N$ (the matrix of off-diagonal blocks)
is nilpotnt. However this is {\it not} the ``S+N"-decomposition of
$T$, for these $S, N$ do {\it not} commute. In case ``S+N"-decomposition
exists, a better matrix representation is obtained by choosing the
off-diagonal blocks to be identity matrices of size $m\times m.$ Such a basis may be constructed starting from an $\E$-basis which gives the usual  ``Jordan block" over $\E,$ and then constructing the corresponding basis over $\F.$   
This was effectively mentioned already by Wedderburn,
cf.  \cite{wed}, and rediscovered by the author early on in this
investigation. The author thanks Rony Gouraige for pointing out the reference \cite{wed}.
\end{remark}

\section{The Affine Case}

In this section we extend the theory to the affine case, and
determine the centralizer of an affine map.

Let $\F$ and $\V$ be as in the introduction, $\A$ the underlying
affine case, and $T = (A, v)$  an affine map which maps $x$ to
$Ax + v$.  As observed there, $A$ in $(A, v)$ has an intrinsic
affine meaning, and $v$ has an intrinsic affine meaning if $A =
I$. Let $S = (\alpha, a), \alpha \in GL(\V)$ be an element of
$GA(\V)$. Then 

\begin{equation}\label{eqn1}
S^{-1} = ({\alpha}^{-1},  - {\alpha}^{-1}a),
\end{equation}

\noindent and

\begin{equation}\label{eqn2}
STS^{-1} = (\alpha A {\alpha}^{-1}, - \alpha A {\alpha}^{-1}a +
\alpha v + a).
\end{equation}

Let ${\mathcal C}_L(\V)$ resp. ${\mathcal C}_A(\V) $ denote the
orbit-spaces $L(\V)/GL(\V)$ resp. $A(\V)/GA(\V)$. For $T$ in
$L(\V)$ resp. $A(\V)$ let $[T]_L$ resp. $[T]_A$ denote its orbit
in ${\mathcal C}_L(\V)$ resp. ${\mathcal C}_A(\V).$ We have seen
that the map $(A, v) \mapsto A$ is a homomorphism  $l: A(\V)
\rightarrow L(\V)$. The formula (6.2) shows that the map $[(A, v)]_A
\mapsto [A]_L$ is a well-defined map $[l]: {\mathcal C}_A(\V)
\rightarrow {\mathcal C}_L(\V)$.   The main result about the map
$[l]$ is

\medskip

\begin{theorem} {$[l]$ is a finite map, that is $[l]^{-1}([A])$ has only
finitely many elements. More precisely,  for $A \in L(\V)$ let
$m_A(x) = (x-1)^rg(x)$, where $g(1) \not= 0$ be its minimal
polynomial. Here $r \ge 0$ is an integer. Then $[l]^{-1}([A])$ has
$r + 1$ elements.}
\end{theorem}

\smallskip

\begin{proof} First consider the generic case where $r = 0$. Consider the
equation $(*) Ax + v = x$ where $x$ is indeterminate, and $A \in
L(\V), v \in \V$ are known entities. Since $r = 0,$ we have $det
(I-A) \neq 0.$  So (*) has a unique solution in $x$. Let $x_0$ be
that unique solution. Let $\tau = (I, x_0).$  Then $\tau (A, v)
{\tau}^{-1} = (A, 0) = A$. So any element in $l^{-1}(A)$ is
conjugate to $A$. It follows that $[l]^{-1}([A])$ has a unique
element.

\smallskip

Now suppose $r > 0.$ Then $\V = \V_1 + \V_2$ (direct sum) where
$\V_1 = ker (A-I)^r,$ and $\V_2 = ker\,  g(A).$ Consider $T = (A,
v).$ Write $v = v_1 + v_2$ where $v_i \in \V_i, i = 1, 2.$ Let
$x_0$ be the solution in $\V_2$ of the equation $(*) Ax + v_2 =
x.$ Such solution exists since $det (I-A)|_{\V_2} \not= 0.$ Let
$\tau = (I, x_0).$  Then $\tau (A, v) {\tau}^{-1} = (A, v_1)$. We
have proved that an element $(A, v) \in l^{-1}(A)$ is in the same
$GA(\V)$-orbit as an element $(A, v_1),$ where  $(A-I)^r(v_1) = 0$. Let $s$
be the least non-negative integer, $s \le r$, such that $(A-I)^s
(v_1) = 0$. Now the theorem follows from the following lemma.

\smallskip

\begin{lemma} Suppose $S= (A, v)$ resp. $T = (A, w)$ be in $A(\V)$ such
that $m_A(x) = (x - 1)^r.$ Let $s$ resp. $t$ be  the least
non-negative integers $\le r$ satisfying $(A-I)^s(v) = 0$ resp.
$(A-I)^t(w) = 0.$ Then $S$ and $T$ are in the same $GA(\V)$-orbit
iff $s = t$.

\end{lemma}

\smallskip

{\it Proof} From (6.2) we see that  $(\alpha, a)$ conjugates $S$ into $T$
iff $\alpha$ is in $Z_L(A)^*$ and $w = (I-A)a + \alpha v$. Since
$m_A(x) = (x - 1)^r$ we are  in the situation of the previous
section. In particular, set $N = I - A,$ and consider the
$Z_L(A)$-invariant refined flag. By symmetry, we may assume $s \le
t.$ In the notation introduced in the previous section, let $\V_t
= \V_{t-1, k}$, and $v$ lies in $\V_t - \V_{t-1, k-1}.$ From the
structure of invertible elements in $Z_L(A)$, we see that $\alpha$
is in $Z_L(A)^*$ and $w = (I-A)a + \alpha v$ iff $s = t$.

\end{proof}

\bigskip

Now we are in a position to determine the centralizer of an affine
map. In effect, we describe a good representative of a
$GA(\V)-$orbit of the centralizer of an affine map.

\smallskip

Let $T = (A, v)$ in $A(\V)$. Let $S = (B, w)$ be in $Z_A(T)$. The
equation $ST = TS$ is equivalent to

i) $BA = AB$, i.e. $B \in Z_L(T).$

ii) $Bv + w = Aw + v$, or $(B - I)v = (A-I)w.$

\smallskip

Case 1)  Assume that $T$ has a fixed point. Then by conjugation by
an element in $GA(\V)$ (or what amounts to the same, by an affine
change of co-ordinates) we may take $v = 0$. With this choice, we
take the flag associated to $A$.   From ii) we see that

$$Z_A(T) = \{(B, w)|  B \in Z_L(A), \, and \, w \in \V_1\}$$

\noindent where $\V_1 = ker \, (A - I).$
\smallskip

Case 2)  Assume that $T$ has no fixed point. Then again by change
of affine co-ordinates by theorem (6.1) we may assume that $m_A(x) =
(x-1)^rg(x), g(1) \not=0$ is the minimal polynomial of $A$, and
$s$ is the least positive integer such that $(A-I)^sv = 0.$ The
equation ii) implies that

$$ (A-I)^s(B - I)v = (B- I)(A-I)^s v = 0 = (A-I)^{s+1}w.$$
So $w$ is  in $\V_{s+1},$ where $\V_i = ker \, (A - I)^i.$

Conversely, suppose that $w$ is  in $\V_{s+1}.$ Then we show that
there exists a $B$ in $Z_L(A)$ such that $(B, w)$ is in $Z_A(T)$,
and we can precisely determine $B$'s having this property. Indeed,
in the double-subscript notation of the flag, $\V_s =\V_{s-1, k}$
(for a suitable $k$), $v$ is in $\V_s - \V_{s-1, k-1}$, and $(A-I)
w$ is in $\V_s$. So there exists $C$ in $Z_L(T)$ so that $Cv =
(A-I) w$, and all such $C$'s can be determined from  the refined
flag. For each such choice of $C$, we can then take $B = C + I.$
These are precisely the $(B, w)$'s in $Z_A(T).$

Notice moreover that $(B, w)$ is in $Z_A(T)^*$ iff $B$ is in
$Z_L(A)^*$. Assume that this is the case, then $Bv$ is in $\V_s -
\V_{s-1, k-1}$. Now equation ii) $(B - I)v = (A-I)w$ shows that
$Bv = v + (A-I)w.$ Since $(A-I)w$ is in $\V_{s-1, k-1}$ we see
that $Bv \equiv v$ mod  $\V_{s-1, k-1}.$ It follows that the
linear map ${\bar B}$ induced by $B$ on $\V_s/\V_{s-1, k-1}$ has
eigenvalue $1$. So $B$ also has eigenvalue $1$, and the $N$-images
of the corresponding eigen-vector show that the multiplicity of
the eigenvalue 1 is at least $s$.

\smallskip

Summarizing, we have proved the following result.

\smallskip

\begin{theorem} {Let $T = (A, v)$ be in $A(\V)$. Let $\V_i = ker \, (A - I)^i.$

1) If $T$ has a fixed point then $Z_A(T)$ is conjugate to
$$ \{(B, w)|  B \in Z_L(A), \, and \, w \in \V_1\}$$

2) If $T$ has no fixed point, then $m_A(x) = (x-1)^rg(x), g(1)
\not=0$ is the minimal polynomial of $A$, and there exists    $s
\le r$  the least positive integer such that $(A-I)^sv = 0.$  Then
$Z_A(T)$ is conjugate to
$$ \{(B, w)|  B \in Z_L(A), \, w \in \V_{s+1}, (B - I)v = (A-I)w.\}$$
An element $(B, w)$ in $Z_A(T)^*$ necessarily has eigenvalue $1$
with multiplicity at least $s$.}
\end{theorem}

\bigskip

\begin{remark} Suppose that $T = (A, v)$ in $A(\V)$ has no fixed
point, the explicit forward orbit-structure of the $T$, or the
orbit structure of $Z_A(T)^*$, is quite complicated, compared to
the neat answer we obtained in case $T$ has a fixed point.
However, next to orbit-structure, for some intuitive
understanding, we can enquire about the invariant sets. On this
score we have some satisfactory information. Namely, if $(B, w)$
is in $Z_A(T)$ then $B$ preserves the refined flag determined by
$A$ in $\U = ker (A- I)^r$, where $m_A(x) = (x-1)^rg(x), g(1)
\not=0,$ is the minimal polynomial of $A$. The affine translates
of each of the subspaces in the flag may be called a family of
{\it affine flags} in $\U$. Clearly $(B, w)$ preserves this family
of affine flags as a whole.

\smallskip

In case the integer $s$ associated to $v$ in $(A, v)$ is $1$, one
can say a bit more. Namely consider the $Z_A(T)$-invariant family
of affine subspaces parallel to $\V_1$. Among these subspaces,
there is actually one $Z_A(T)$-invariant subspace. Namely, up to
an affine change of co-ordinates we may assume that $v$ is
actually an eigenvector of $A$. Then the eigen-space $\V_1 = ker \, (A-I)$ itself
is $Z_A(T)$-invariant.

\end{remark}
\bigskip

\begin{remark}Consider the case $\F = \R,$ the field of real
numbers, or $\F = \C$ the field of complex numbers. On the Lie
algebra level, one may ask for ``normal forms" of the solutions of
affine vector fields on $\A.$ This amounts to solutions of the
ODEs

$$\frac{dx}{dt} = Ax + v, \, A \in L(\V), v \in \V.$$

\noindent Equivalently one may ask for normal forms of representatives of
conjugacy classes of one-parameter subgroups of $GA(\V).$ In the texts
on ODEs, cf. for example \cite{earl}, this ODE is solved by the method of
variation of parameters. The ideas in this section provide a
short-cut. Namely consider the affine map $(A, v).$ If this map has a
fixed point, (which is the case if $det(A-I) \not= 0$), then by an
affine change of coordinates we can make $v = 0,$ and the
solutions are orbits of the one parameter group $t \mapsto e^{tA}$
in the new coordinate system. If the map has no fixed point then $A$
must have eigenvalue 1. Write $m_A(x) = (x-1)^rg(x), g(1)\not=0.$
Let $\R^n = \V = \V_1 + \V_2,$ where $\V_1 = ker\, (A- I)^r, \V_2
= ker \, g(A).$ Let $v = v_1 + v_2,$ where $v_i$ is in $\V_i$ for
$i = 1, 2.$ By an affine change of coordinates we can make $v_2 =
0.$ Choose the least positive integer $s$ such that $(A-I)^sv_1 =
0.$ Then in the new coordinate system, the solutions are orbits of
the one-parameter group
$$t \mapsto (e^{tA},
tv_1 + \frac{t^2}{2!} Av_1 + \frac{t^3}{3!} A^2v_1 + \ldots +
\frac{t^s}{s!} A^{s-1}v_1) $$ The point is that one can always
make the translational part of a one-parameter group of the affine
group a polynomial, rather than an infinite series, in $t$, either
by conjugacy in $GA(\V)$, or what is the same, by an appropriate
affine change of coordinates. This ``normal form" of a
one-parameter group indicates that its orbits, or the orbits of
its centralizer, in $\V$ are more complicated than in the linear
case, when there is an ``unavoidable" translational part, which
carries an affine meaning.

\end{remark}

\begin{remark} From a computational, or  algorithmic, perspective the
decomposition $\R^n = \V = \V_1 + \V_2,$ is readily computatble.
The main issue is the computation of $e^{tA}.$ Now $m_A(x)$ is
algorithmically computable as the last non-zero diagonal entry in
the Smith normal form of the characteristic matrix $xI - A.$
Assume that we have a factorization of $m_A(x)$ into its
irreducible factors. When $\F = \R$ the irreducible factors are of
degree 1 or 2. The (generalized) eigenspaces corresponding to linear
factors and the corresponding refined lattice of
$Z_L(A)$-invariant subspaces, the corresponding (Jordan) canonical
forms and their exponentials are all algorithmically computatble.
When $\F = \R$ and $m_A(x)$ has irreducible factors of degree 2,
again the corresponding refined lattice of $Z_L(A)$-invariant
subspaces is algorithmically computable. However the suggested
rational canonical form in the texts of algebra using the
companion matrix of an irreducible factor is not useful for
computation of the exponential. If the irreducible factor is $x^2
- 2ax + b, a^2 - b < 0,$ then its companion matrix is $ \left[
     \begin{array}{cc}
       0 & -b \\
  1 & 2a \\
              \end{array}
   \right]
 $. It is decisively better to use the matrix $ \left[
     \begin{array}{cc}
       a & -c \\
  c & a \\
              \end{array}
   \right]
 $, $a^2 + c^2 = b$ which is conjugate to the companion matrix. For then its
exponential becomes readily computable, namely,  $e^a \left[
     \begin{array}{cc}
       \cos c & -\sin c \\
  \sin c & \cos c \\
              \end{array}
   \right].
 $ Also one
should use the (forgotten) rational form as explained in  section 6,
where the non-diagonal blocks are $2\times 2$ identity matrices.
This is indicated in the texts and exercises in \cite{earl} and \cite{hirsh}, without
adequate explanation.

\end{remark}

\section{Paramatrization Theorems}

As stated in the introduction, we have interpreted the phrase
``understanding the dynamics" in our set-up to mean the parametrizations of similarity classes,
$z$-classes, and finally the elements in  $L(\V)$ and  $A(\V)$
themselves in terms of objects having significance independent of
the choices of linear or affine co-ordinate systems. 
Here the word ``parametrization" is used in the following sense.
The sets $L(\V)$ and $A(\V)$ are the ``unknown" sets which we wish
to understand in terms of the ``known" sets $\F$ and $\V$, and the
``universally known" sets such as natural numbers, integers,
rational numbers, and if one wishes, also real and complex
numbers, and any other similar sets, and the sets derived from
such sets by applying the allowable constructions in the model of
``naive" set theory. Loosely speaking, the parameters having
values in abelian groups are called ``numerical parameters", and
the others, such as decompositions into subspaces or flags, are
called ``spatial" parameters. In more abstract terms they are made
precise in theorem 2.1 of \cite{kulkarni}.

\smallskip

The parametrizations that are obtained here are    
in terms of the ``arithmetic" of $\F$ as reflected
in the monic irreducible polynomials, and subspaces of $\V.$ The
datum of irreducible polynomial in $\F[x]$ of degree $m$ is
equivalent to the datum of a simple field extension $\E$ of $\F$
such that $[\E : \F] = m,$ and a primitive element $\alpha$ of
$\E$ over $\F$. Starting with a monic irreducible polynomial $p(x)
\in \F[x]$ we have $\E = \F[x]/(p(x)),$ and $\alpha = [x],$ the
class of $x$ in $\F[x]/(p(x))$. Conversely, given $(\E, \alpha),$
we get $p(x)$ as the minimal polynomial of $\mu_{\alpha}$ where
$\mu_{\alpha}: \E \rightarrow \E, \mu_{\alpha}(u) = \alpha u.$
Here $\mu_{\alpha}$ is regarded as a $\F$-linear map of the
vector space $\E$ over $\F.$ To be completely precise, to obtain a
one-to-one correspondence between $p(x)$ and pairs $(\E, \alpha)$
we need to consider the $\F$-isomorphism classes of $(\E,
\alpha)$'s. Namely, the pairs $(\E_1, \alpha_1), (\E_2, \alpha_2),$ are
$\F$-isomorphic if there exists an $\F$-isomorphic if there exists an 
$\F$-isomorphism carrying $ \alpha_1$ to  $ \alpha_2$. In particular if  we 
fix $\E$ in its isomorphism class of field extensions of $\F$, then $\alpha$ is defined only up to the action of $G(\E/\F)$, the group of
$\F$-automorphisms of $\E.$

\smallskip

Let $n = dim\, \V$ and $\pi: n = \sum_{i=1}^{r} n_i$ be a
partition of $n$. A {\it decomposition ${\mathcal D}_{\pi}$
patterned on the partition $\pi$} of $\V$ is a direct sum
decomposition $\V = \oplus_{i=1}^{r} \V_i$ into subspaces, where
$dim\, \V_i = n_i.$

\smallskip

Let $n = dim\, \V.$ Let $m$ be a divisor of $n$, and $n = ml.$ Let
$r$ be a natural number, and $r$ pairs of natural nubers $\{(s_1,
\sigma_1), (s_2, \sigma_2), \ldots (s_r, \sigma_r)\}$ such that
$s_1 < s_2 < \ldots < s_r,$ and $l = \Sigma_{i=1}^{r} s_i
\sigma_i$. A {\it flag} of type $(n, m; \{(s_1, \sigma_1), (s_2,
\sigma_2), \ldots (s_r, \sigma_r)\})$ is an increasing family of
subspaces $\V_{i, j}, 0 \le i \le r, 0 \le j $ such that the
successive quotients have dimensions: $(m\sigma_r, m\sigma_{r-1},
\ldots ,m\sigma_1)$ occurring $s_1$ times, $(m\sigma_r,
m\sigma_{r-1}, \ldots , m\sigma_2)$, occurring $s_2-s_1$ times,
$\ldots (m\sigma_r, m\sigma_{r-1})$ occurring $s_{r-1} - s_{r-2}$
times, $(m\sigma_r)$ occurring $s_r - s_{r-1}$ times.

\smallskip

Such a  flag is denoted by $\mathcal F((n, m; \{(s_1, \sigma_1),
(s_2, \sigma_2), \ldots (s_r, \sigma_r)\})$

\smallskip

Now suppose that there exists a simple field extension $\E$ of
$\F$ such that $[\E : \F] = m.$   Then since the dimension of each
successive sub-quotient of $\mathcal F(n, m; \{(s_1, \sigma_1),
(s_2, \sigma_2), \ldots (s_r, \sigma_r)\})$ is divisible by $m$,
 it has a structure of a vector space over $\E$. We denote
such a choice of  an $\E$-structure, a bit loosely, by
$J_{\E}$. When we wish to emphasize the sub-quotient $\W$ we shall
specify $J_{\E, \W}$. The choices of $J_{\E, \W}$'s are by no
means unique. In fact $GL(\W)$ clearly acts on the set of
$\E$-structures on $\W$. An important point, which is easy to see,
is that the action of $GL(\W)$ on  the set of $\E$-structures is
{\it transitive}.

\smallskip

Now we define an important  notion of {\it compatibility} of
$J_{\E, \W}$'s. In the flag, we have special components $\V_i,
1\le i \le s_r = d.$ The {\it compatibility} of $J_{\E, \W}$'s for
the successive sub-quotients in the flag means that there   are
$\E$-structures on $\V_{i+1}/\V_i, 0 \le i < d$ such that for all
the components $\V_{i, j}$ in the chain from $\V_i$ to $\V_{i+1}$
the sub-quotients $\V_{i, j}/\V_i$ are $\E$-subspaces of
$\V_{i+1}/\V_i,$ and the $\E$-structure on $\W = \V_{i, j}/\V_{i,
j-1}$ coincides with $J_{\E, \W}$. One may enquire whether the
$\E$-structures on $\V_{i+1}/\V_i,$ are similarly compatible with
a single $\E$-structure on $\V$. As discussed in section 5,  
this turns out to be a subtle point related to the existence 
of ``$S+N$"-decomposition.  

\smallskip

With this preparation,  we are in a position to describe our
parametrizations.

\vskip .5in

\begin{theorem} 1: A) A $GL(\V)$-orbit in its action on $L(\V)$ is
parametrized by the following data.

i) A {\it primary} partition $\pi: n = \sum_{i=1}^{r} n_i$, $n_i =
m_il_i.$

ii) The {\it secondary} partitions $l_i = \sum_{j=1}^{r_i} s_{i,
j} \sigma_{i, j},$ where $s_{i,1} < s_{i,2} < \ldots s_{i,r_i}.$

iii) An $\F$-isomorphism class of pairs $(\E_i, \alpha_i)$, where
$\E_i$ is a simple field extension of $\F$ of degree $m_i$ with $\alpha_i$ as its primitive element, for $i = 1, 2, ...., r$.

\medskip

 B)  A $GA(\V)$-orbit in its action on $A(\V)$ is
parametrized by the data i), ii), iii) as in A) and with $m(x) =
(x-1)^ug(x), g(1) \not= 0,$

\smallskip

iv) A non-negative integer $s \le u.$

\end{theorem}
\vskip .5in

\begin{theorem}:  A) A $z$-class  in  the $GL(\V)$-action on $L(\V)$ is parametrized
by the following data.

i) A {\it primary} partition $\pi: n = \sum_{i=1}^{r} n_i$, $n_i =
m_il_i,$

ii) The {\it secondary} partitions $l_i = \sum_{j=1}^{r_i} s_{i,
j} \sigma_{i, j},$ where $s_{i,1} < s_{i,2} < \ldots s_{i,r_i}.$

iii) Simple field extensions $\E_i, 1\le i \le r$ of $\F$, $[\E_i
: \F_i ] = m_i.$

\medskip

B)  A $z$-class  in  the $GA(\V)$-action on $A(\V)$ is
parametrized by the data i), ii), iii) as in A). In the case of a
$GA(\V)$-orbit-class of $(A, v)$ has $m_A(x) = (x-1)^ug(x), g(1)
\not= 0,$ then there is an additional parameter

\smallskip

iv) A non-negative integer $s \le u.$

\end{theorem}
\vskip .5in

\begin{theorem} A) An element of $L(\V)$  is uniquely determined by the
following data. The data i), ii) , iii) of part A) in theorem 7.1,
in particular the field extensions $\E_i = \F[x]/(p_i(x)),$ and
the primitive elements $\alpha_i$.

iv) A decomposition ${\mathcal D}_{\pi}: \V = \oplus_{i=1}^{r}
\V_i$ of $\V$ patterned on the primary partition $\pi$.

v) Flags ${\mathcal F}((n_i, m_i; \{(s_{i, 1}, \sigma_{i, 1}),
(s_{i, 2}, \sigma_{i, 2}), \ldots (s_{i, r_i}, \sigma_{i, r_i})\})$ of
subspaces   in ${\V}_i,$ patterned on the secondary partitions.

vi) Compatible $\E_i$-structures on the sub-quotients in the flag
in each $\V_i$.

\medskip

B) An element $T$ of $A(\V)$  is uniquely determined by the
following data.

\smallskip

 Case 1. ($T$ has a fixed point): Choose a fixed point as the
origin. So $T$ may be identified with an element in $L(\V)$. The
data i), ... , vi) in part A) is independent of the choice of the
fixed point. These data and the affine subspace of fixed points
determine $T$.

\smallskip

Case 2. ($T$ has no fixed point): Express $T$ as $(B, v)$ so
that there exists $s$ a least positive integer such that $(I -
B)^s v = 0.$ Then the  invariants i), ... , vi) in part A)
associated to $B$ and $v$ uniquely determine $T$.

\end{theorem}
\bigskip

The proofs of theorems 7.1-7.3 are given in the next two sections.
A major consequence of theorem 7.2, cf. also section 10,  is the 
following theorem.

\begin{theorem}{Let $\V$ be an $n$-dimensional vector space over a field $\F$. 
Suppose $\F$ has the property that there are only finitely many extensions 
of $\F$ of degree at most $n$. Then there are finitely many $z$-classes 
of $GL(\V)$-, resp. $GA(\V)$-, actions on $L(\V)$, resp. $A(\V)$.}
\end{theorem}

 \bigskip
 \section{Proof of Parametrization Theorems 7.1 and  7.3}

We begin with the proof of Theorem 7.1. Notice that the data in ii), and iii) in part A) is just   the numerical data regarding the exponents and multiplicities in the
elementary divisors in the classical theory, which can be independently read from the refined flag. Given an element $T$ in $L(\V)$, we associate to it

i) the minimal polynomial $m(x) = m_T(x) = \Pi_{i=1}^{r}p_i(x)^{d_i},$

ii) the  primary partition $dim\, \V = \sum_{i=1}^{r} dim \, \V_i$
where $\V_i = ker\, p_i(T)^{d_i},$ and

iii) the secondary partitions with $s_{i, j}$'s being the
exponents in the elementary divisors $p_i(x)^{s_{i, j}}$s, and
${\sigma}_{i, j}$s being the multiplicities of $p_i(x)^{s_{i,
j}}$s.

Conversely suppose we have the data i), ii), iii). We first show
that there actually exists $T$ in $L(\V)$ which realizes this
data, and secondly that any two elements in $L(\V)$ having the
same data are in the same $GL(\V)$-orbit.

Take an arbitrary decomposition  $\V = \oplus_{i=1}^{r}\V_i$
patterned over the primary partition. Next construct an
appropriate flag in each $\V_i$ with type given by the pairs
$(s_{i, j}, \sigma_{i, j})$'s. Let $\E_i = \F[x]/(p_i(x)),$ and
$\alpha = [x].$ Equip the sub-quotients in the flag in $\V_i$ with
a compatible family of $\E_i$-structures. Take an arbitrary
$\E_i$-basis $(e_1, e_2, \ldots , e_k)$ in the component $\V_{0,
1}$ of the flag. (We have actually $k = \sigma_{s_r}.$) Then

$$(e_1, \alpha e_1, \alpha^2 e_1, \ldots, \alpha^{m_i-1}e_1, e_2,
\alpha e_2, \ldots \ldots, \alpha^{m_i-1}e_k)$$ is an $\F$-basis
of $\V_{0, 1}$. Moreover we can define the operator $T$ on $\V_{0,
1}$ which is multiplication by $\alpha$. We can continue this
process to all the components in the chain ending in $\V_1,$ and
define the operator $T$ on $\V_1$ having the minimal polynomial
$p(x)$. Next we consider the component $\V_{1, 1}$ in the flag.
Notice that by construction $dim_{\F} \V_{1, 1}/\V_1$ is $m_ik$,
and $\V_{1, 1}/\V_1$ has an $\E_i$-structure. Choose $(e'_1, e'_2,
\ldots , e'_k)$ in $\V_{1, 1}$ whose classes $[e'_i]$ modulo
$\V_1$ form an $\E_i$-basis. Define $T^je'_u, 1\le j \le m-1, 1\le
u \le k$ in $\V_{1, 1}$ so that their classes $[T^je'_u]$ modulo
$\V_1$ are $[\alpha^j e_u].$ Now a crucial point is to define
$p(T)e'_i = e_i$ in $\V_{0, 1},$ and more generally $p(T)T^je'_i =
T^je_i, 1\le j \le m_i-1.$ It is easy to see that continuing this
process along the successive components in the flag we obtain a
basis of $\V_i$ and an operator $T$ in $L(\V_i)$ having the given
secondary partition on $\V_i.$ Taking the direct sum we obtain an
operator $T$ on $\V$ having the minimal polynomial $m(x)$ and the
given primary and secondary partitions.

Finally suppose that $T, T'$ are two elements in $L(\V)$ having
the same data. Then the dimension of a primary component $\V_i$
equals $m_il_i.$ (Here $l_i$ is the largest power of $p_i(x)$
dividing the characteristic polynomial.) So by appropriate
conjugation by an element of $GL(\V)$ we may suppose that both $T$
and $T'$ have the same primary components $\V_i$s. So we reduce to
the case where $m_T(x) =  m_{T'}(x) = p(x)^d$, where $p(x)$ is a
monic irreducible in $\F[x].$ Next by  hypothesis  $T, T'$ have
the same secondary partitions. Then we can construct the flags and
the bases $e_j$'s, $e'_j$'s of $\V$ adapted to the respective flags.
Then the element $g \in GL(\V), ge_i \mapsto e'_i $ conjugates $T$
into $T'.$

This finishes the proof of part A) of theorem 7.1. The proof of part
B) can be completed along the same lines using the results in section 6.

As for the proof of Theorem 7.3, observe that the data the isomorphism class of $(\E, \alpha)$ determines an irreducible polynomial in $\F[x].$ So the proof may
be completed along the lines of theorem 7.1.

\section{Proof of the Parametrization Theorem 7.2}

Let $S, T$ be in the same $z-$class in $L(\V)$. This means that
$Z_L(S)^*$ and $Z_L(T)^*$ are conjugate by an element $u$ in
$GL(\V)$. First we show that this implies that $Z_L(S)$ and
$Z_L(T)$ are conjugate, in fact by the same element $u,$ in
$GL(\V)$. This follows from the following lemma.

\smallskip

\begin{lemma}{Let $T$ be in $L(\V)$. Then $Z_L(T)$ as an $\F$-subalgebra 
of $L(\V)$ and $Z_L(T)^*$ as a subgroup of $GL(\V)$ uniquely determine each other.}
 \end{lemma}

\begin{proof} Indeed $Z_L(T)$ determines $Z_L(T)^*$ as the multiplicative 
subgroup of its units. Conversely let $S$ be a non-invertible element in $Z_L(T).$ 
Then $m_S(x) = x^kf(x),$ with $k > 0,$ and $f(0) \not= 0.$
Let $\V_0 = ker \, S^k,$ and $\V_1 = ker \, f(S).$ So $\V = \V_0 \oplus \V_1$ 
is a $T$-invariant decomposition. For any such decomposition, let $J_{\V_0, \V_1}$ 
denote the operator which is identity on $\V_0,$ and zero on $\V_1.$ 
Then $J_{\V_0, \V_1}$ is in $Z_L(T)$, and $S_1 = S + J_{\V_0, \V_1}$
is clearly in $Z_L(T)^*.$ Thus $Z_L(T)$ is a linear span of $Z_L(T)^*$
and the operators $J_{\V_0, \V_1}$ corresponding to all
$T$-invariant decompositions $\V = \V_0 \oplus \V_1.$ This proves that 
$Z_L(T)^*$ determines $Z_L(T).$ 

\end{proof}

\begin{remark} Although the following observation is not needed in the 
proof that follows, the above lemma raises a question whether 
$Z_L(T)$ itself is always a linear span of $Z_L(T)^*.$ This is 
indeed the case if $\F$ has more than two elements. For indeed, 
let $S,  \V_0,$ and  $\V_1$ be as in the above proof. Let
$c$ be an element in $\F$ different from $0$ and $1$. Define $U_1$
as $(S - I)|_{\V_0}$ on $\V_0,$ and $cS|_{\V_1}$ on $\V_1.$ Define
$U_2$ as  $I|_{\V_0}$ on $\V_0,$ and $(1-c)S|_{\V_1}$ on $\V_1.$
Then $U_1, U_2$ are in $Z_L(T)^*$ and $S = U_1 + U_2$. Thus in
fact an element in $Z_L(T)$ is a sum of at most two elements in
$Z_L(T)*$. 

\end{remark}

\begin{remark}The restriction that $\F$ has more than
two elements in the above remark is a genuine one. For example consider
an $n$-dimensional vector space $\V$ over $\F_2,$ the field with two elements. Assume $n \ge 2.$ Let $T$ be an operator with
$m_T(x)\, = \, x^k(x-1)^l, $ where $k \ge 1, \, l \ge 1.$   Consider 
the $T$-invariant decomposition $\V = \V_0 \oplus \V_1,$ where $\V_0 = ker \, T^k,$ and $\V_1 = ker \, (T - I)^l.$ Then $Z_L(T) = Z_L(T|_{\V_0}) \times Z_L(T|_{\V_1}).$ Clearly $Z_L(T|_{\V_0}) = 
\F[T|_{\V_0}],$ and $Z_L(T|_{\V_1}) = \F[T|_{\V_1}],$ whereas 
$Z_L(T)^*$ consists of $f(T)$ where $f(0) = 1.$ If we take the sum of 
{\it even} number of elements of $Z_L(T)^*$ then we get an operator all of whose eigenvalues are $0.$ On the other hand if we take the sum of {\it odd} number of elements of $Z_L(T)^*$ then we get an operator all of whose eigenvalues are $1.$ It follows $T$ cannot be written as a sum of elements of $Z_L{(T)}^*$.

\end{remark}

\bigskip

In view of the lemma we can assume that $Z_L(S)$ and $Z_L(T)$ are
conjugate by an element $u$ in $GL(\V)$. Replacing $S$ by
$uSu^{-1}$ we may assume that $Z_L(S) = Z_L(T)$

\smallskip

Let $C$ be the center of $Z_L(T).$ By the Frobenius' bicommutant theorem, we have $C = \F[S] =
\F[T].$ It is important to note that $C$ does not determine $T$.
However every element of $C$ leaves every $T-$invariant (or
$S-$invariant) subspace invariant. Let $p_i(x)$ be the primes
associated to $T,$ and $\V = \oplus \V_i$ the corresponding
primary decomposition. Let $\W$ be a $T$-invariant subspace of
$\V_i$ such that the pair $(\W, T|_{\W})$ is dynamically
equivalent to $(\F[x]/(p_i(x)^d), \mu_x).$ Then $\W_j = ker \,
p_i(x)^j, 0 \le j\le d,$ are precisely all the $T$-invariant
subspaces of $\W.$ Since a subspace of $\V$ is $T$-invariant iff
it is $S$-invariant, it follows that $\W_j$'s are precisely also
all the $S$-invariant subspaces of $\W$. It follows that
$m_{S|_{\W}}(x)$ must be of the form $q(x)^e$ where $q(x)$ is a
monic irreducible polynomial in $\F(x).$

Next note that the same $q(x)$ works for every $T$-invariant
subspace $\U$ such that the pair $(\W, T|_{\W})$ is dynamically
equivalent to $(\F[x]/(p_i(x)^e), \mu_x)$ for some $e$. For there
exists an operator $A$ in $Z_L(T) = Z_L(S)$ which maps $\W$ onto
$\U$ equivariantly with the action of $S.$ It follows that $\V =
\oplus \V_i$ is also a primary decomposition with respect to $S$.
So in particular, $n = \sum_in_i, dim\, \V_i = n_i$ is a
well-defined choice of a primary partition of $n$. Now restrict
the action of $Z_L(S) = Z_L(T)$ to $\V_i.$ For the same reason we
see that the refined flag, and in particular the secondary
partitions are well-defined invariants of $Z_L(S) = Z_L(T),$ which
are independent of the choices of a $T$ with the property $C =
\F[T].$   Finally considering the action of $Z_L(S) = Z_L(T)$ on
$\V_{d_i}/\V_{d_{i-1}}$ we see that the simple field extension
$\E_i = \F[x]/(p_i(x))$ is a well-defined invariant of $Z_L(S) =
Z_L(T).$

Conversely, given the the primary and secondary partitions and the field extensions $\E_i$'s of appropriate degree there clearly exists an operator having this data, and its orbit class is uniquely determined. This finishes the proof of the theorem 7.2 in the linear case. Using the results of section 6, the proof can be extended to the affine case. 
 
\section{Generating Functions for $z$-classes}

Let $\mathcal D$ be a collection of extension fields of finite degree of a field $\F$ with the property that $\mathcal D$ contains only finitely many extensions of a given degree.
Significantly, this property is automatically satisfied for the collection of {\it all} extension fields in the following cases: 1) $\F$ algebraically closed, 2)$\F = \R,$ 3) $\F = $ a local field, 4) $\F = $ a finite field. A case of arithmetic interest is 5) $\F = \Q$, $S = $ a finite set of primes, and $\mathcal D = $ the collection of all extension fields obtained by adjoining all $n$-th roots of all primes in $S.$ From the parametrization theorem 7.2, it follows that for any such collection $\mathcal D$, and for a fixed $n,$ there are only finitely many $z$-classes of linear maps on an $n$-dimensional vector space over $\F$ with the extension fields in $\mathcal D$. So one can form a generating function $$\mathcal Z_{\F, \mathcal D}(x) = \Sigma_{n= 0}^{\infty} z(n)x^n.$$ As is expected from the parametrization theorems, these functions are closely related to the generating functions for partitions. One may also consider the restricted generating functions which enumerate the $z$-classes of dynamically semi-simple operators, or cyclic operators. In both cases the secondary partitions have simple types. (Recall that a pair $(\V, T)$ is {\it cyclic} if there exists a vector $v$ such that $V = \F[T]v.$ Clearly 
$(\V, T)$ is {\it cyclic}  iff $deg \, m_T(x) = dim \, \V.$) We denote the corresponding generating functions by $$\mathcal Z_{\F, \mathcal D, s}(x)\; {\rm and} \; \mathcal Z_{\F, \mathcal D, c}(x)$$ respectively.

Let $\Pi_n$ denote the set of all partitions of $n$, and $p(n)$ the cardinality of $\Pi_n$.  A partition $\pi$ of $n$ with signature
$(1^{a_1}2^{a_2} \ldots n^{a_n})$ is the partition in which $i$ occurs $a_i$ times, so $n = \Sigma_ia_ii.$

Let $f(x) = 1 + \Sigma_{n=1}^{\infty} b(n)x^n$ be a  formal power series. To $f(x)$ we associate a new power series, $\mathcal P(f(x)) = 1+ \Sigma_{n=1}^{\infty} c(n)x^n.$   Here $c(n)$ is a sum  $\Sigma_{\pi \in \Pi_n}c_{\pi},$ where $c_{\pi} = {\Pi_{i=1}^{n}}b(i)^{a_i}$, if $\pi$ has signature $(1^{a_1}2^{a_2} \ldots n^{a_n}).$ Notice that the well-known Eulerian generating function for partitions $P(x) = 1 + \Sigma_{n=1}^{\infty} p(n)x^n$ is $\mathcal P(g(x))$ where $g(x) = \Sigma_{n=0}^{\infty} x^n$ is the geometric series.

{\it The Absolute Case}:  Here $\F$ is algebraically closed. Here $\mathcal D$ consists of a single element, namely $\F$ itself, and we omit its mention. First consider the easy cases of i) semisimple operators, or ii) cyclic operators. In both cases, the secondary partitions are uniquely determined. Let the {\it primary} partition of an operator $T$ be $\pi: n = \sum_{i=1}^{r} n_i.$ If $T$ is  semisimple then the secondary partitions of $n_i$s have signatures $(1^{n_i}).$ If $T$ is cyclic then the secondary partitions of $n_i$s have signatures $(n_i^1).$ So 

$$\mathcal Z_{s}(x) = \mathcal Z_{c}(x) = P(x).$$

\noindent On the other hand consider the case of all $z$-classes. Let $T$ be the operator whose primary partition has signature $(1^{a_1}2^{a_2} \ldots n^{a_n}).$ Then the number of secondary partitions associated with this partition is $p(1)^{a_1}p(2)^{a_2} \ldots p(n)^{a_n}.$ It follows that 

$$\mathcal Z(x) = \mathcal P(P(x)) = \Pi_{k=1}^{\infty}\frac{1}{1 - p(k)x^k}.$$

\begin{theorem} {$\mathcal Z(x)$ is a meromorphic function on the unit disc, and it cannot be extended beyond the unit disc.}
\end{theorem}

\begin{proof} A simple estimate for $p(n)$ is $p(n) \le e^{K{\sqrt{n}}},$ for $K > 0,$ cf. \cite{C}, ch. VII, section 3. It follows that $p(n)^{\frac{1}{n}}$ tends to $1$ as $n$ tends to infinity. So for {\hbox{$|x| < 1, \; Z(x)$}} defines a meromorphic function on the unit disc. Its poles are at 
{\hbox{$x = {{p(n)}^{-\frac{1}{n}}}e^{{\frac{2k\pi i}{n}}}$}}, for $n = 1, 2, \ldots$ and $0\le k \le n$. So it also follows that the function cannot be extended meromorphically beyond the unit disc.  
\end{proof}

It appears that this type of generating function has not appeared in number theory before.

Notice that if we consider more generally the case of an arbitrary field $\F$, but restrict to $\mathcal D = \{\F\}$, then we get the same generating functions.

\bigskip

{\it General  Case}: Let $\E$ be an element of $\mathcal D,$ and $[\E ; \F] = m.$ Clearly the contribution to $\mathcal Z_{\mathcal D}(x)$ coming from $\E$ is $\mathcal Z(x^m).$ We denote this contribution by 
$\mathcal Z_{\F, \E}(x)$. Clearly 

 $$\mathcal Z_{\F, \mathcal D}(x) = \Pi_{\E \in \mathcal D} \mathcal Z_{\F, \E}(x).$$
Since we have assumed that $\mathcal D$ contains only finitely many extensions of a given degree, this product is well-defined. Two notable cases are i) $\F = \R,$ the field of real numbers,  and ii)
$\F = \F_q$, the finite field with $q$ elements, and $\mathcal D$ consists of all extensions of finite degree.  Then 

 $$\mathcal Z_{\R, \mathcal D}(x) = \mathcal Z(x) \mathcal Z(x^2).$$

 $$\mathcal Z_{\F_q, \mathcal D}(x) = \Pi_{n=1}^{\infty}\mathcal Z(x^n).$$

\end{document}